\documentclass[a4paper,12pt]{amsart}
\usepackage{amssymb,amsmath}
\usepackage{hyperref}
\newcommand{\Jac}{{\rm Jac}}
\newcommand{\Q}{\mathbb{Q}}
\newcommand{\F}{\mathbb{F}}
\newcommand{\Proj}{\mathbb{P}}
\newcommand{\cl}[1]{{\mathcal #1}}
\newcommand{\Z}{\mathbb{Z}}
\newcommand{\ep}{\varepsilon}
\newtheorem{theorem}{Theorem}
\newtheorem{corollary}{Corollary}
\newtheorem{prop}{Proposition}
\newtheorem{lemma}{Lemma}

\newcommand{\HH}{{\rm H}}

\begin{document}
\title{Counting Rational Points on Cubic Curves}
\author[D.R. Heath-Brown \and D. Testa]{D.R. Heath-Brown \and D. Testa\\ \\Mathematical Institute, University of Oxford}
\begin{abstract}
We prove upper bounds for the number of rational points on non-singular cubic curves defined over the rationals.  
The bounds are uniform in the curve and involve the rank of the corresponding Jacobian.  The method used in the 
proof is a combination of the ``determinant method'' with an $m$-descent on the curve.
\end{abstract}
\address{\begin{minipage}{250pt}
Mathematical Institute \\
24--29 St.\ Giles' \\
OX1 3LB \\
Oxford, UK \\
\end{minipage}}
\date{}

\thanks{The second author was supported EPSRC grant number EP/F060661/1.}
\email{\begin{minipage}[t]{250pt}
rhb@maths.ox.ac.uk,\\testa@maths.ox.ac.uk 
\end{minipage}}
\maketitle
\section{Introduction and Statement of Results}

Let $F(X_0,X_1,X_2)\in\Z[X_0,X_1,X_2]$ be a non-singular
cubic form, so that $F=0$ defines a smooth plane cubic 
curve $C$ in $\Proj^2$. This
paper will be concerned with upper bounds for the counting function
\[N(B)=N(F;B)=\#\{P\in C(\Q):\, H(P)\le B\},\]
where the height function $H(P)$ is defined as
$\max \{ |x_0|,|x_1|,|x_2| \}$ when $P=[x_0,x_1,x_2]$ with coprime integer
values of $x_0,x_1,x_2$.  We are interested in obtaining upper bounds
for $N(B)$ which are uniform with respect to the curve $C$, or in which the
dependence on $C$ is explicit.

Providing that $C(\Q)$ is non-empty, we can view
$C$ as an elliptic curve and use the machinery of canonical heights.
When the rank $r$ is positive we will have
\begin{equation}\label{neron}
N(B)\sim c_F (\log B)^{r/2}
\end{equation}
as $B\rightarrow\infty$, as was shown by N\'{e}ron.
On the other hand, if $r=0$ we know that
$N(B)\le 16$ by Mazur's theorem~\cite{mazur} on torsion groups of
elliptic curves.  While this latter result is of course uniform over
all $F$, the estimate~(\ref{neron}) is certainly not.  However
Heath-Brown~\cite{asterisque} investigated what might be proved
uniformly, and showed that
\[
N(B) \ll \bigl( B \| F \| \bigr)^{Ar/\lambda}
\]
for some absolute constant $A$.
Here $\|F\|$ is the maximum modulus of the coefficients of $F$, and
$\lambda =\log N$, where $N$ is the conductor of the Jacobian ${\rm
  Jac}(C)$.  (This result comes from
combining the fourth and fifth displayed formulae 
on page~22 of~\cite{asterisque} with the third display on page~24.)  Indeed the
result may be simplified by calling on Theorem~4 of Heath-Brown's
work~\cite{ann}, which implies the following.
\begin{prop} \label{trenta}
For a plane cubic curve $C$ defined by a primitive integer form $F$, 
either $N(B)\le 9$ or $\|F\|\ll
B^{30}$.
\end{prop}
We therefore deduce that
\begin{equation}\label{b1}
N(B)\ll B^{Ar/\lambda}
\end{equation}
with a new absolute constant $A$.
It should be emphasized that this is completely uniform in the curve
$C$.  One expects that the ratio $r/\lambda$ tends to zero as
the conductor $N$ tends to infinity.  This is the ``Rank
Hypothesis'' of~\cite{asterisque}, which would follow if one knew the
truth of both the Generalized Riemann Hypothesis for the $L$-functions
of elliptic curves and the Birch--Swinnerton-Dyer Conjecture.  If we
assume the Rank Hypothesis, then we can deduce that
\[N(B)\ll_{\ep}B^{\ep}\]
for any fixed $\ep>0$, with an implied constant independent of the
curve $C$.

For the case in which the rank $r$ is small one can do rather better
by inserting a result of David~\cite[Corollary~1.6]{david} into this
analysis.  Thus one can show that
\begin{equation}\label{sd}
N(B) \ll (\log B)^{1+r/2}
\end{equation}
with complete uniformity.  
Therefore the non-uniform asymptotic formula~(\ref{neron}) may be replaced
by a uniform upper bound, at the expense of a power of 
$\log B$ only.
We prove~(\ref{sd}) in the appendix to
this paper.  The exponent may easily be replaced by $c+r/2$ with a
constant $c<1$, but our methods do not allow us to remove it entirely
when $r$ is small.

There is a second suite of results on $N(B)$ which have their origin
in work of Bombieri and Pila~\cite{BP}.  The latter was concerned with
general affine curves, and was adapted by 
Heath-Brown~\cite[Theorem~3]{ann} to handle projective curves.  For a projective
plane curve of arbitrary degree $d$ one obtains
\[
N(B)\ll_{d,\ep}B^{2/d+\ep},
\]
so that $N(B)=O_{\ep}(B^{2/3+\ep})$ in our case.  These results are
essentially best possible for general curves, in the sense that 
one has $N(B)\gg B^{2/d}$ for the singular curve $X_0^{d-1}X_1=X_2^d$.  
If one takes account of the height $\|F\|$ of the form $F$ (defined as
the maximum modulus of the coefficients of $F$), one can do slightly 
better, showing that
\[
N(B)\ll_{d,\ep}B^{\ep} \bigl( B^{2/d}\|F\|^{-1/d^2}+1 \bigr)
\]
if the form $F$ is primitive.  Thus for the cubic case one has
\begin{equation}\label{hsave}
N(B) \ll_{\ep} B^{\ep} \bigl( B^{2/3}\|F\|^{-1/9}+1 \bigr).
\end{equation}
These
results are due to Ellenberg and Venkatesh~\cite[Proposition~2.1]{EV}
(and independently, in unpublished work, to Heath-Brown).  The above
bounds allow a small saving over the exponent $2/d$ unless $B$ is
large compared with $\|F\|$, while in the remaining case an argument based on
the group structure of ${\rm Jac}(C)$ gives a sharper bound.  Thus
Ellenberg and Venkatesh~\cite[Theorem~1.1]{EV} were able to show that
for each degree $d$ there was a corresponding $\delta(d)>0$ such that
\[
N(B)\ll_d B^{2/d-\delta(d)}
\]
providing that the curve 
$F(X_0,X_1,X_2)=0$
is non-rational.  This is nice, since it shows that when $C$ has
positive genus $N(B)$ is
distinctly smaller than for the genus zero curve $X_0^{d-1}X_1=X_2^d$.  
In the case
of non-singular cubic curves, which is our concern here, Ellenberg and 
Venkatesh showed~\cite[\S 4.1]{EV} that
\[N(B)\ll B^{2/3-1/405}.\]
In unpublished work Salberger has given a rather different approach,
which replaces $1/405$ by $2/327$ in certain cases.

In summary then, we have two general approaches for smooth cubic
curves.  The first
uses the group structure on the corresponding elliptic curve, and has bad
uniformity in $C$, while the second applies to arbitrary
curves, and is therefore (almost) restricted to the exponent $2/d$
which one has for rational curves.  The result of the present paper
may be thought of as interpolating between these two types of result.
We shall prove the following theorem.
\begin{theorem}\label{one}
Let 
\[F(X_0,X_1,X_2)\in\Z[X_0,X_1,X_2]\]
be a non-singular
cubic form, so that $F=0$ defines a smooth plane cubic 
curve $C$.  Let $r$ be the rank of $\Jac(C)$.  Then
for any $B\ge 3$ and any positive integer $m$ we have
\[
N(B)\ll m^{r+2} \Bigl( \log^2B+B^{2/(3m^2)}\log B \Bigr)
\]
uniformly in $C$, with an implied constant independent of $m$.
\end{theorem}

Taking $m=1+[\sqrt{\log B}]$ we have the following immediate
corollary.
\begin{corollary}
Under the conditions above we have
\[N(B)\ll (\log B)^{3+r/2}\]
uniformly in $C$.
\end{corollary}
While this is slightly weaker than~(\ref{sd}) it is interesting to
observe that the results of David~\cite{david} are obtained by very
different methods, based on transcendence theory.

We can combine Theorem~\ref{one} with~(\ref{hsave}) if we have a
suitable bound for $r$.  This 
is discussed by Ellenberg and Venkatesh~\cite[pages~2177 \& 2178]{EV}, 
who show that
\[2^r\le \|F\|^{6+o(1)}.\]
Thus, taking $m=2$ in Theorem~\ref{one}, we have
\[N(B)\ll \|F\|^{6+o(1)}B^{1/6+o(1)}.\]
On comparing this with~(\ref{hsave}) we see that the worst case is
that in which $\|F\|=B^{9/110}$, and that in every case we can save
$B^{1/110}$.  We conclude as follows.
\begin{theorem}\label{two}
Let $\delta<1/110$.  Then for any smooth plane cubic curve $C$ we have
\[N(B)\ll B^{2/3-\delta}\]
uniformly in $C$.
\end{theorem}

\section{The Descent Argument}

Let $\psi \colon C \times C \to \Jac(C)$ be the morphism defined by $\psi (P,Q) = [P]-[Q]$.  
Let $m$ be a positive integer and define an equivalence relation $\sim_m$
on $C(\Q)$ by saying that $P\sim_m Q$ if and only if
$\psi(P,Q) \in m \bigl( \Jac(C)(\mathbb{Q}) \bigr)$.  The number of equivalence classes 
is at most $16m^r$, allowing for possible torsion in $\Jac(C)(\mathbb{Q})$. If $K$
is an equivalence class, we write $N_K(B)$ for the number of
points $P\in K$ with $H (P)\le B$, 
so that there is a class $K^*$ for which
\begin{equation}\label{Kb}
N(B) \ll m^r N_{K^*}(B).
\end{equation}
We proceed to estimate $N_K(B)$
for a given class $K$.  If we fix a point $R$ counted by $N_K(B)$
then for any other point $P$ counted by $N_K(B)$ there will be a
further point $Q\in C(\Q)$ such that $[P]=m[Q]-(m-1)[R]$.

Having fixed $R$ we define the curve $X=X_R$ by 
\begin{equation} \label{ics}
X_R := \bigl\{ (P,Q)\in C \times C : \, [P]=m[Q]-(m-1)[R] \bigr\}
\end{equation}
in $\Proj^2\times \Proj^2$.  This allows us to write
\[
N_K(B) = \# \bigl\{ (P,Q) \in X(\Q) :\, H(P)\le B \bigr\} .
\]
The fundamental idea here is that we are counting rational points 
$(P,Q)$ on the curve $X \subset \Proj^2\times \Proj^2$.  It might 
seem natural to work with the point $Q\in C(\Q)$ alone.  However 
it is hard to control $H (Q)$ sufficiently accurately, and working 
with $P$ and $Q$ simultaneously avoids this difficulty.

None the less we do need a crude bound for $H (Q)$, which the following
lemma provides.
\begin{lemma}\label{crude}
For any $c>0$ there is a constant $A$ depending only on $c$, with 
the following property.
Let $C$ be a smooth plane cubic curve defined by a primitive form $F$ with 
$\|F\| \le cB^{30}$, and let $R$ be a point in $C(\Q)$.  Suppose 
that $(P,Q)$ is a point in $X_R(\Q)$ and that $B\ge 3$.  Then if 
$H(P),H(R)\le B$ we have
\[
H(Q)\le B^A.
\]
\end{lemma}

\begin{proof}
We use the well known fact that we can choose a model for
$\Jac(C)$ in Weierstrass normal form 
$y^2=x^3+\alpha x+\beta $ such that
\[
h ([1,\alpha , \beta ])\ll 1+\log \|F\|
\]
and so that
\[
h_x(\psi(P,R))\ll 1+\log H(P) +\log H(R)+\log\|F\|
\]
and
\[
\log H(P)\ll 1+h_x(\psi(P,R))+\log H(R)+\log\|F\|,
\]
where $h_x$ is the logarithmic height of the $x$-coordinate.  
We shall also use the fact that on $\Jac(C)$ the canonical height 
$\hat{h}$ satisfies 
\[
|\hat{h}(R)-h_x(R)|\ll 1+h([1 , \alpha , \beta ])\ll 1+\log \|F\|.
\]

Since $\psi(P,R)=[P]-[R]=m([Q]-[R])=m\psi(Q,R)$ we deduce that
$\hat{h}(\psi(P,R))=m^2\hat{h}(\psi(Q,R))$, whence
\begin{eqnarray*}
\log H(Q)&\ll& 1+h_x(\psi(Q,R))+\log H(R)+\log\|F\|\\
&\ll& 1+\hat{h}(\psi(Q,R))+\log H(R)+\log\|F\|\\
&=& 1+m^{-2}\hat{h}(\psi(P,R))+\log H(R)+\log\|F\|\\
&\ll& 1+m^{-2}h_x(\psi(P,R))+\log H(R)+\log\|F\|\\
&\ll& 1+\log H(P)+\log H(R)+\log\|F\|\\
&\ll& \log B+\log\|F\|\\
&\ll &\log B
\end{eqnarray*}
since $\|F\|\ll B^{30}$.  The lemma then follows.
\end{proof}

\section{Outline of the Determinant Method}

In this section we shall set up the ``determinant method'', following
the ideas laid down by Heath-Brown~\cite[\S 3]{ann}, but modified to
handle a bi-homogeneous curve.  We shall only do as much
as is needed for our application, but it will be clear, we hope, how
one might proceed in more generality if necessary.  In view of
Proposition~\ref{trenta} we shall assume that $\|F\| \ll B^{30}$.

We take $p$ to be a prime of good reduction for $C$.  For each point 
$Q'$ on $C(\F_p)$ we define the set
\[
S(Q';p,B) = 
\Bigl\{ (P,Q)\in X_R(\Q) :\, H(P)\le B,\,\, \overline{Q}=Q' \Bigr\} ,
\]
where $\overline{Q}$ denotes the reduction from $C(\Q)$ to $C(\F_p)$.
We proceed to estimate $\# S(Q';p,B)$, for a
particular choice $Q'$, bearing in mind
that there are $O(p)$ possible points $Q'$.  
In view of~(\ref{Kb}) there exist $K^*$ and $Q^*$ such that
\begin{equation}\label{KQstb}
N(B) \ll m^r N_{K^*}(B) \ll m^r p \, \#S(\overline{Q^{*}};p,B).
\end{equation}
It will be convenient to label
the points in $S(\overline{Q^{*}};p,B)$ as $(P_j,Q_j)$ for $1\le
j\le N$, say, and to write $P_j$ and $Q_j$ in terms of primitive integer
triples as $P_j=[p_{0j},p_{1j},p_{2j}]$ and
$Q_j=[q_{0j},q_{1j},q_{2j}]$.

We now fix degrees $a,b\ge 1$ and consider a set of bi-homogeneous
monomials of bi-degree $(a,b)$ of the form
\begin{equation}\label{mon}
x_{0}^{e_0}x_{1}^{e_1}x_{2}^{e_2}y_{0}^{f_0}y_{1}^{f_1}y_{2}^{f_2},
\end{equation}
with 
\[
e_0+e_1+e_2=a\;\;\;\mbox{and}\;\;\; f_0+f_1+f_2=b.
\]
The exponent vectors $(e_0,e_1,e_2;f_0,f_1,f_2)$ will run over a
certain set $\cl{E}$, which we now describe.  Let $I_1$ be the
$\Q$-vector space of all bi-homogeneous forms
in $\Q[x_0,x_1,x_2,y_0,y_1,y_2]$ with
bi-degree $(a,b)$, and let $I_2$ be the subspace of such forms which
vanish on $X$.  Since the monomials form a basis for $I_1$ there is a
subset of monomials in $I_1$ whose corresponding cosets form a basis 
for $I_1/I_2$.  We choose $\cl{E}$ to correspond to the set of 
monomials forming such
a set of representatives.  The following result tells us the
cardinality of $\cl{E}$.

\begin{lemma}\label{DT}
If $a,b$ and $m$ are positive integers satisfying 
$\frac{1}{a}+\frac{m^2}{b}<3$, then $\#\cl{E}=3(m^2a+b)$.
\end{lemma}

We shall prove this later, in~\S \ref{DTproof}.  We shall assume
henceforth that $a \geq 1$ and $b\ge m^2$ for convenience, which 
will suffice to ensure that $\#\cl{E}=3(m^2a+b)$.

We proceed to construct a matrix $M$ whose entries are the monomials
\[
p_{0j}^{e_0} \; p_{1j}^{e_1} \; p_{2j}^{e_2} \; q_{0j}^{f_0} \; q_{1j}^{f_1} \; q_{2j}^{f_2}
\]
with exponents in $\cl{E}$.  The row of the matrix $M$ indexed by $j$
corresponds to the point $(P_j,Q_j)\in S(\overline{Q^{*}};p,B)$; the columns 
of $M$ correspond to exponent vectors in $\cl{E}$.
Thus $M$ is an integer matrix of size
$N\times E$, where $E:=\# \cl{E}$.

We will show that if the prime $p$ and the degrees $a$ and $b$ are
appropriately chosen, then the rank of $M$ is strictly less than $E$.
It will follow that there is a non-zero column vector $\b{c}$ such
that $M\b{c}=\b{0}$.  The entries of $\b{c}$ are indexed by the monomials
in $\cl{E}$, and we therefore produce a bi-homogeneous form $G$, say,
with bi-degree $(a,b)$, such that
$G(p_{0j},p_{1j},p_{2j};q_{0j},q_{1j},q_{2j})=0$ for each $j\le N$.
Thus the points $(P_j,Q_j)$ all lie on the variety
$Y\subset\Proj^2\times\Proj^2$ given by $G=0$.
Our choice of exponents in $\cl{E}$ ensures that the irreducible curve $X$
does not lie wholly inside $Y$.  Thus $X\cap Y$ has components of
dimension 0 only, and we deduce that
\[
N \le \#(X\cap Y) \le X \cdot (a,b) = 3(m^2a+b) ,
\]
where the intersection number computation is explained in the proof of Lemma~\ref{geo}.
This gives us a bound for $\# S(\overline{Q^*};p,B)$, and hence
by~(\ref{KQstb}) also for $N(B)$.  

We summarize our findings as follows.
\begin{lemma}\label{Nest}
Let $p$ be a prime of good reduction for $C$.  Suppose we have
integers $a\ge 1$ and $b\ge m^2$ such that the matrix $M$ above
necessarily has rank strictly less than $E$.  Then
\[
N(B) \ll m^rp(m^2a+b) 
\]
with an absolute implied constant.
\end{lemma}

\section{Vanishing Determinants}

In order to show that the matrix $M$ considered above has rank
strictly less than $E$ we may clearly suppose that $N\ge E$.  Under
this assumption we will show that each $E\times E$ minor of $M$ 
vanishes.  Let $\Delta $ be the $E \times E$ matrix formed 
from $E$ rows of $M$.  Our strategy, as in Heath-Brown~\cite[\S 3]{ann}, 
is to estimate the (archimedean) size of $\det \Delta$, and to
compare it with its $p$-adic valuation.

The archimedean estimate is easy.  According to Lemma~\ref{crude},
there is an absolute constant $A$ such that every entry in the
matrix $\Delta$ has modulus at most $B^a \cdot B^{Ab}$.  Since
$\Delta$ is an $E\times E$ matrix, we conclude as follows.
\begin{lemma}\label{arch}
There is an absolute constant $A$ such that
\[|\det \Delta|\le E^EB^{E(a+Ab)}.\]
\end{lemma}

The $p$-adic estimate forms the core of the determinant method.
We remark at once that if we choose different projective
representatives
\[(p_{0j}',p_{1j}',p_{2j}')=\lambda_j(p_{0j},p_{1j},p_{2j})\] 
for $P_j$, this will not affect the value of $v_p(\det \Delta)$, providing
that $\lambda_j$ is a $p$-adic unit; and similarly for $Q_j$.

In general, the point $P$ is determined by $Q$ via the relation 
\[
[P]=m[Q]-(m-1)[R] 
\]
appearing in~(\ref{ics}). Since $p$ is a prime of good reduction 
the map which takes $Q$ to $P$ is well-defined 
over both $\Q$ and
$\F_p$. 
If $Q=[(q_0,q_1,q_2)]$ the map
is given by $P=[(p_0,p_1,p_2)]$ with forms 
\[p_0(y_0,y_1,y_2), 
p_1(y_0,y_1,y_2),p_2(y_0,y_1,y_2)\in\Z[y_0,y_1,y_2]\]
of equal degree.  Different
points $Q$ may require different forms $p_0,p_1,p_2$.  
Let $Q^*=[(q_0^*,q_1^*,q_2^*)]$ be as in~(\ref{KQstb}), with $q_0^*,q_1^*,q_2^*$ 
integers not all divisible by $p$.
Since the map is well defined over $\F_p$ at $\overline{Q^*}$ there is a
choice of forms such that 
\[\bigl( p_0(q_0^*,q_1^*,q_2^*),p_1(q_0^*,q_1^*,q_2^*),
p_2(q_0^*,q_1^*,q_2^*) \bigr) \not \equiv (0,0,0) \pmod{p}.\]
With this particular choice we find that if $(P_j,Q_j)\in
S(\overline{Q^{*}};p,B)$ then $P_j=[(p_{0j},p_{1j},p_{2j})]$ with 
$p_{ij}=p_i(q_{0j},q_{1j},q_{2j})$ and 
\[
(p_{0j},p_{1j},p_{2j})\not\equiv (0,0,0)\pmod{p} .
\]
By the remark above, this choice of projective
representative for $P_j$ does not affect $v_p(\det \Delta)$.

Because $q_0^*,q_1^*,q_2^*$ are not all divisible by $p$, 
we suppose, without loss of
generality, that $p\nmid q_0^*$.   Since $\overline{Q_j}=\overline{Q^{*}}$
for all the pairs $(P_j,Q_j)$ under consideration we may think of
$Q_j\in\Proj^2(\Q_p)$ as $[(1,z_{1j},z_{2j})]$ with $z_{1j}=
q_{1j}q_{0j}^{-1}$ and $z_{2j}=q_{2j}q_{0j}^{-1}$ both $p$-adic
integers.  By the remark made earlier, replacing $(q_{0j},q_{1j},q_{2j})$ by
$(1,z_{1j},z_{2j})$ and replacing 
\[
(p_{0j},p_{1j},p_{2j}) = \bigl( p_0(q_{0j},q_{1j},q_{2j}), \,
p_1(q_{0j},q_{1j},q_{2j}), \, p_2(q_{0j},q_{1j},q_{2j})\bigr)
\]
by
\[
\bigl( p_0(1,z_{1j},z_{2j}), \, p_1(1,z_{1j},z_{2j}), \, p_2(1,z_{1j},z_{2j}) \bigr)
\]
does not affect $v_p(\det \Delta)$.

With these changes, we have replaced the original matrix $\Delta$ 
by a matrix $\Delta_0$ whose $i$-th column contains values
$g_i(z_{1j},z_{2j})$ for $1\le j\le E$, where $g_i(x,y)$ is a
polynomial in $\Z_p[x,y]$.  We proceed to write $z_1$ as a function of
$z_2$, which will enable us to replace the polynomials $g_i$ by
functions of $z_{2j}$ alone.

To do this we begin by showing that 
\begin{equation}\label{partials}
\frac{\partial F}{\partial y_1} (q_0^*,q_1^*,q_2^*) 
\;\;\; \mbox{and} \;\;\; 
\frac{\partial F}{\partial y_2} (q_0^*,q_1^*,q_2^*)
\end{equation}
cannot both be divisible by $p$.
If they were, we would
have
\begin{eqnarray*}
q_0^*\frac{\partial F}{\partial y_0} (q_0^*,q_1^*,q_2^*) &\equiv&
(q_0^*,q_1^*,q_2^*) \cdot \nabla F(q_0^*,q_1^*,q_2^*)\\
&=&3F(q_0^*,q_1^*,q_2^*)\\
&\equiv& 0\pmod{p}.
\end{eqnarray*}
Since $p\nmid q_0^*$ this would yield $p\mid \partial
F(q_0^*,q_1^*,q_2^*)/\partial y_0$, but then we would have $p\mid \nabla
F(q_0^*,q_1^*,q_2^*)$, which is impossible.  Thus at least one of
the derivatives~(\ref{partials}) must be coprime to $p$.  With no loss of
generality we assume that 
\[p\nmid \frac{\partial F}{\partial y_1} (q_0^*,q_1^*,q_2^*) .\]

We are now ready to
apply Lemma~5 of Heath-Brown~\cite{ann}, which is a form of the
$p$-adic Implicit Function Theorem.  For any positive integer $n$ this 
produces a polynomial $f_n(t)\in\Z_p[t]$ such that if
$\overline{Q_j}=\overline{Q^{*}}$ then 
\[z_{1j}\equiv f_n(z_{2j})\pmod{p^n}.\]
Substituting $f_n(z_{2j})$ for $z_{1j}$ in $\Delta_0$ we obtain a
matrix $\Delta _n$ with 
\[
\Delta_0\equiv\Delta_n\pmod{p^n}
\]
in which 
\[
(\Delta_n)_{ij}=h_i(z_{2j})
\]
for appropriate polynomials $h_i(t)\in\Z_p[t]$.  Lemma~6 of Heath-Brown~\cite{ann} 
now shows that 
\[
p^{E(E-1)/2}\mid \det \Delta_n ,
\]
since~\cite[(3.6)]{ann} yields $f=E-1$.  Choosing $n=E(E-1)/2$, 
we therefore conclude as follows.
\begin{lemma}
If $p$ is a prime of good reduction for $C$, then $p^{E(E-1)/2}$
divides $\det \Delta$.
\end{lemma}

Comparing this result with Lemma~\ref{arch}, we see that $\Delta$ must
vanish, providing that
\[
p>E^{2/(E-1)}B^{2(a+Ab)/(E-1)}.
\]
We note that $E^{2/(E-1)}<4$ for $E>2$.  Moreover, since we are
assuming that $\|F\| \ll B^{30}$, the discriminant $D_F$ of $F$ is at most a
power of $B$.  The number of primes of bad reduction is then at most 
\[
\omega(6 |D_F|) \ll \frac{\log|D_F|}{\log\log|D_F|}\ll
\frac{\log B}{\log\log B}.
\]
where $\omega (n)$ denotes the number of prime divisors of $n$.  
However if $P$ is sufficiently large there are at least $P/(2\log P)$
primes between $P$ and $2P$.  Thus there is an absolute constant, 
$c_0$ say, such
that any range $P<p\le 2P$ with $P\ge c_0\log B$ contains a prime $p$
of good reduction. We take 
\[
P=c_0\log B+4B^{2(a+Ab)/(E-1)}
\]
and deduce, for a suitable choice of $p$, that $\Delta=0$.  We then
conclude from Lemma~\ref{Nest} that
\[
N(B) \ll m^r(m^2a+b)\{\log B+B^{2(a+Ab)/(E-1)}\}.
\]

It remains to choose $a$ and $b$.  We recall that $E=3(m^2a+b)$, 
where $a\ge 1$ and $b\ge m^2$.  We shall in fact take $b=m^2$ and
$a=1+[\log B]$, whence
\[
\frac{2(a+Ab)}{E-1} \le \frac{2(a+m^2A)}{3m^2a} = 
\frac{2}{3m^2} + O \bigl( (\log B)^{-1} \bigr).
\]
We therefore deduce that
\[
N(B)\ll m^{r+2} \Bigl( \log^2 B+B^{2/(3m^2)}\log B \Bigr) ,
\]
as required for Theorem~\ref{one}.

\section{Proof of Lemma~\ref{DT}}\label{DTproof}

Recall that for any point $T \in C$ we define 
\[
X_T := \Bigl\{ (P,Q)\in C \times C :\, [P]=m[Q]-(m-1)[T] \Bigr\} 
\]
in $\Proj^2 \times \Proj^2$ (see~(\ref{ics})).

\begin{lemma} \label{geo}
Let $T$ be a point of $C$ and suppose that $a$, $b$ and $m$ are
positive integers satisfying the inequality 
$\frac{1}{a}+\frac{m^2}{b}<3$.  Then the restriction of global sections 
\[
\HH^0 \bigl( \Proj^2 \times \Proj^2 , \mathcal{O}_{\Proj^2 \times
  \Proj^2} (a,b) \bigr) \to \HH^0 \bigl( X_T , \mathcal{O}_{X_T} (a,b)
\bigr)\]
is surjective and the dimension of $\HH^0 \bigl( X_T , \mathcal{O}_{X_T} (a,b) \bigr)$ is $3(m^2a+b)$.  
\end{lemma}

\begin{proof}
We make repeated use of the following standard reasoning.  
Suppose that $Y$ is a variety, that 
$D \subset Y$ is an effective divisor on $Y$, and that 
$\mathcal{L}$ is a line bundle on $Y$.  There is a 
short exact sequence 
\begin{equation} \label{restri}
0 \to \mathcal{L} (-D) \to \mathcal{L} \to \mathcal{L}|_{D} \to 0
\end{equation}
of sheaves on $Y$.  From the long exact cohomology sequence 
associated to~(\ref{restri}), we deduce that 
the restriction of global sections 
\[
\HH^0(Y , \mathcal{L}) \to \HH^0(D , \mathcal{L}|_{D})
\]
is surjective if 
the cohomology group $\HH^1(Y , \mathcal{L}(-D))$ vanishes.  In our 
argument, the variety 
$Y$ is always a subvariety of $\Proj^2 \times \Proj^2$ and the line 
bundle $\mathcal{L}$ is the restriction 
to $Y$ of the line bundle $\mathcal{O}_{\Proj^2 \times \Proj^2}(a,b)$ 
of bi-homogeneous polynomials 
of bi-degree $(a,b)$.  The vanishing of $\HH^1(Y , \mathcal{L}(-D))$ 
is a consequence of the 
Kodaira Vanishing Theorem (\cite[Remark~7.15]{Ha}) or of the 
Kawamata--Viehweg Vanishing Theorem (\cite{Ka,Vi}).

\noindent
{\bf Reduction one: from $\mathbb{P}^2 \times \mathbb{P}^2$ to $C \times \mathbb{P}^2$.}
The ideal of functions on $\mathbb{P}^2 \times \mathbb{P}^2$ that
vanish on $C \times \mathbb{P}^2$ is generated by the degree three homogeneous
polynomial $F$ in the coordinates of the first factor $\mathbb{P}^2$.  Thus among the functions bi-homogeneous
of bi-degree $(a,b)$ on $\mathbb{P}^2 \times \mathbb{P}^2$, the ones vanishing on $C \times \mathbb{P}^2$ are the
functions bi-homogeneous of bi-degree $(a-3,b)$.  Therefore the sequence~(\ref{restri}) becomes
$$
0 \to \mathcal{O}_{\mathbb{P}^2 \times \mathbb{P}^2}(a-3,b) \to \mathcal{O}_{\mathbb{P}^2 \times \mathbb{P}^2}(a,b) \to 
\mathcal{O}_{C \times \mathbb{P}^2} (a,b) \to 0
$$
in this case.  The vanishing of $\HH^1(\mathbb{P}^2 \times \mathbb{P}^2 , \mathcal{O}_{\mathbb{P}^2 \times \mathbb{P}^2}(a-3,b))$ 
under the assumptions $a>0$ and $b>-3$ follows by the Kodaira Vanishing Theorem.

\noindent
{\bf Reduction two: from $C \times \mathbb{P}^2$ to $C \times C$.}
We argue as above to find that if the inequalities $a>0$ and $b>0$ hold, then 
the cohomology group $\HH^1(C \times \mathbb{P}^2 , \mathcal{O}_{C \times \mathbb{P}^2}(a,b-3))$ vanishes.

\noindent
{\bf Reduction three: from $C \times C$ to $X_T$.}
The curve $X_T$ is a divisor on $C \times C$ and the
sequence~(\ref{restri}) becomes 
$$
0 \to \mathcal{O}_{C \times C} \bigl( (a,b) - X_T \bigr) \to \mathcal{O}_{C \times C}(a,b) \to \mathcal{O}_{X_T} (a,b) \to 0
$$
in this case.  The surface $C \times C$ is an abelian surface and therefore every effective divisor on $C \times C$ is 
nef.  Hence, the vanishing of the group $\HH^1 \bigl( C \times C , \mathcal{O}_{C \times C} \bigl( (a,b) - X_T \bigr) \bigr)$ is a 
consequence of the Kawamata-Viehweg Vanishing Theorem if the inequalities $(0,1) \cdot \bigl( (a,b) - X_T \bigr) > 0$ and 
$\bigl( (a,b) - X_T \bigr) ^2 > 0$ hold.  We have 
\begin{itemize}
\item $(1,0) \cdot (0,1) = 9$ since a general line in $\Proj^2$ intersects $C$ in three points; 
\item $(1,0) ^2 = (0,1) ^2 = 0$ since a general pair of lines in $\Proj^2$ intersects in a point not on $C$;
\item $(1,0) \cdot X_T = 3m^2$ since for a fixed point $P$ on $C$ there are $m^2$ pairs $(P,Q)$ in $X_T$;
\item $(0,1) \cdot X_T = 3$ since for a fixed point $Q$ on $C$ there is a unique pair $(P,Q)$ in $X_T$;
\item $(X_T)^2 = 0$ since for all $T, T' \in C$ the curves $X_T$ and $X_{T'}$ are algebraically equivalent and 
if $(m-1)([T]-[T']) \neq 0$, then the curves $X_T$ and $X_{T'}$ are disjoint.
\end{itemize}
Thus the equalities 
$$
(0,1) \cdot \bigl( (a,b) - X_T \bigr) = 3 (3a - 1)
$$
and 
$$
\bigl( (a,b) - X_T \bigr) ^2 = 6ab \left( 3 - \frac{1}{a} - \frac{m^2}{b} \right)
$$
hold, and the first part of the lemma follows.

To compute the dimension of $\HH^0 \bigl( X_T,\mathcal{O}_{X_T} (a,b) \bigr)$ we observe that the projection of the curve 
$X_T \subset C \times C$ onto the second factor is an isomorphism,
and hence the curve $X_T$ is smooth of genus one.  Moreover, using the intersection numbers computed above, the 
line bundle $\mathcal{O}_{X_T} (a,b)$ on $X_T$ has degree 
\[
X_T \cdot (a,b) = 3(m^2a+b)>0.
\]
It follows that the group 
$\HH^1 \bigl( X_T,\mathcal{O}_{X_T} (a,b) \bigr)$ vanishes and we 
therefore conclude from the Riemann-Roch formula 
that the dimension of $\HH^0 \bigl( X_T,\mathcal{O}_{X_T} (a,b)
\bigr)$ 
is $3(m^2a+b)$.
\end{proof}

\begin{proof}[Proof of Lemma~\ref{DT}]
We use the notation introduced in the discussion above the statement of Lemma~\ref{DT}.  The projection map $I_1 \to I_1/I_2$ 
corresponds to the restriction of global sections 
$$
\HH^0 \bigl( \Proj^2 \times \Proj^2 , \mathcal{O}_{\Proj^2 \times \Proj^2} (a,b) \bigr) \to \HH^0 \bigl( X , \mathcal{O}_X (a,b) \bigr) . 
$$
By Lemma~\ref{geo} the vector space $I_1/I_2$ has dimension $3(m^2a+b)$ and is spanned by bi-homogeneous monomials of 
bi-degree $(a,b)$.  This suffices for the lemma.
\end{proof}

\appendix
\section*{Appendix}

Our goal in this appendix is to prove the following result.
\begin{prop}\label{P2}
For any smooth plane cubic curve $C$ we have
\[ N(B)\ll (\log B)^{1+r/2}\]
with an absolute implied constant, where $r$ is the rank of ${\rm Jac}(C)$.
\end{prop}

For the proof we adapt the arguments of Heath-Brown 
\cite[\S 4]{asterisque}.  According to our Proposition~\ref{trenta} it
will suffice to assume that $\|F\|\ll B^{30}$.  Thus the fourth and
fifth displayed formulae on page~22 of~\cite{asterisque} and
the third displayed formula on page~23 produce a bound
\[
N(B) \ll \# \bigl\{ (n_1,\ldots,n_r)\in\Z^r:\, Q(n_1,\ldots,n_r)\le c\log B \bigr\}
\]
for a suitable absolute constant $c$, where $Q$ is the quadratic form
corresponding to the canonical height function on ${\rm Jac}(C)$.  We
now call on Corollary~1.6 of David~\cite{david}, which shows that if
$D$ is the discriminant of ${\Jac}(C)$ then the
successive minima $M_j$ of $\sqrt{Q}$ satisfy 
\[
M_1\gg (\log|D|)^{-7/16},\;\;\;
M_2\gg (\log|D|)^{-1/6},\;\;\;
M_3\gg (\log|D|)^{-7/96},
\]
\[
M_4\gg (\log|D|)^{-1/40},\;\;\;\mbox{and}\;\;\;
M_5\gg (\log|D|)^{1/240}.
\]
Note that David's result refers to the successive minima for $Q$, while
we have given the corresponding results for $\sqrt{Q}$.

To estimate the number of integer vectors with $Q(n_1,\ldots,n_r)\le
R$, say, we may apply Lemma~1 of Davenport~\cite{Dav}, with the
distance function $\sqrt{Q}$.  This yields the bound
\[
\# \bigl\{ (n_1,\ldots,n_r)\in\Z^r:\,  Q(n_1,\ldots,n_r)\le R \bigr\} \le
\prod_{j\le r}\max \left\{ 1, 4\frac{\sqrt{R}}{M_j} \right\}.
\]
In our case we deduce that
\begin{equation}\label{app1}
N(B) \ll \prod_{j\le r} \max \left\{ 1 , 4\frac{\sqrt{c\log B}}{M_j} \right\} .
\end{equation}
In view of our lower bound for $M_5$ there is an absolute constant
$D_0$ such that
\[
4\frac{\sqrt{c\log B}}{M_j}\le 4\frac{\sqrt{c\log B}}{M_5} \le \sqrt{\log B}
\] 
when $|D|\ge D_0$ and $j\ge 5$.  When $r\ge 4$ it then follows that
\begin{eqnarray*}
N(B)&\ll& 
(\log B)^{(r-4)/2}\prod_{j\le 4} \max \left\{1, \frac{\sqrt{\log B}}{M_j}\right\} \\
&\ll& (\log|D|)^{7/16+1/6+7/96+1/40}(\log B)^{r/2}
\end{eqnarray*}
for $|D|\ge D_0$.  When $|D|\le D_0$ the rank $r$ is bounded and the 
same result follows at once from~(\ref{app1}).  Finally, since $D$ is 
bounded by a power of $B$ and 
\[
\frac{7}{16}+\frac{1}{6}+\frac{7}{96}+\frac{1}{40}<1
\]
we deduce that 
\[
N(B)\ll(\log B)^{1+r/2},
\]
as required.

We remark that one may show in the same way that if $r$ is
sufficiently large ($r\ge 244$ say) then the clean result
\[
N(B)\ll(\log B)^{r/2}
\]
holds.

\end{document}